\documentclass[12pt]{article}
\usepackage{amsfonts}
\usepackage{amsmath}
\usepackage{amssymb}
\usepackage{amsthm}
\usepackage{enumerate}
\usepackage{bm} 

\newtheorem{theorem}{Theorem}[section]
\newtheorem{lemma}{Lemma}[section]

\newtheorem{fact}{Fact}[section]

\theoremstyle{definition}

\begin{document}

\begin{center}
\vskip 1cm{\LARGE\bf{An Arithmetic Function Arising from the Dedekind $\psi$ Function} 
\vskip 1cm
\large
Colin Defant\\
Department of Mathematics\\
University of Florida\\
United States\\
cdefant@ufl.edu}
\end{center}
\vskip .2 in

\begin{abstract}
We define $\overline{\psi}$ to be the multiplicative arithmetic function that satisfies \[\overline{\psi}(p^{\alpha})=\begin{cases} p^{\alpha-1}(p+1), & \mbox{if } p\neq 2; \\ p^{\alpha-1}, & \mbox{if } p=2 \end{cases}\]  for all primes $p$ and positive integers $\alpha$. Let $\lambda(n)$ be the number of iterations of the function $\overline{\psi}$ needed for $n$ to reach $2$. It follows from a theorem due to White that $\lambda$ is additive. Following Shapiro's work on the iterated $\varphi$ function, we determine bounds for $\lambda$. We also use the function $\lambda$ to partition the set of positive integers into three sets $S_1,S_2,S_3$ and determine some properties of these sets. 
\end{abstract} 

\bigskip

\noindent 2010 {\it Mathematics Subject Classification}:  Primary 11A25; Secondary 11B83. 

\noindent \emph{Keywords: Iterated function; Dedekind function; additive function.} 
\section{Introduction} 
Throughout this paper, we let $\mathbb{N}$, $\mathbb{N}_0$, and $\mathbb{P}$ denote the set of positive integers, the set of nonnegative integers, and the set of prime numbers, respectively. For any prime $p$ and positive integer $x$, we let $v_p(x)$ denote the exponent of $p$ in the prime factorization of $x$.
\par 
The study of iterated arithmetic functions has played an important role in number theory over the past century. In particular, Euler's totient function has received considerable attention due to the fact that iterates of this function eventually become constant. In 1943, H. Shapiro studied a function $C(n)$ that counts the number of iterations of Euler's totient function needed for $n$ to reach $2$, and he showed that this function is additive \cite{shapiro43}. Many scholars have since built upon Shapiro's work. For example, White found a class of functions related to Euler's totient function which have the property that the appropriate analogues of Shapiro's function $C$ are additive \cite{white62}. Our goal is to study the iterates of one such function which is related to the Dedekind $\psi$ function.  
\par 
The Dedekind $\psi$ function is a multiplicative arithmetic function that satisfies $\psi(p^{\alpha})=p^{\alpha-1}(p+1)$ for all primes $p$ and positive integers $\alpha$. This function, which Dedekind introduced in order to aid in the study of modular functions, has found numerous applications in various areas of mathematics, especially those closely related to group theory \cite{hampejs14, shimura71}. We will define $\overline{\psi}$ to be the multiplicative arithmetic function that satisfies \[\overline{\psi}(p^{\alpha})=\begin{cases} p^{\alpha-1}(p+1), & \mbox{if } p\neq 2; \\ p^{\alpha-1}, & \mbox{if } p=2 \end{cases}\]  for all primes $p$ and positive integers $\alpha$. One may wish to think of the function $\overline{\psi}$ as a sort of hybrid of the Euler $\varphi$ function and the Dedekind $\psi$ function.  Indeed, $\overline{\psi}$ maps $p^{\alpha}$ to $\psi(p^{\alpha})$ when $p$ is odd and maps $2^{\alpha}$ to $\varphi(2^{\alpha})=2^{\alpha-1}$. Observe that if $n>2$, then $\overline{\psi}(n)$ is even just as $\varphi(n)$ is even. The reason for modifying the $\psi$ function to create the function $\overline{\psi}$ in this manner is to ensure that iterates of any positive integer $n$ eventually become constant. We will make this notion more precise after some definitions.  
\par 
For any function $f\colon\mathbb{N}\rightarrow\mathbb{N}$, we define $f^0(n)=n$ and $f^k(n)=f(f^{k-1}(n))$ for all positive integers $n$ and $k$. We refer to the sequence $n,f(n),f^2(n),\ldots$ as the \emph{trajectory of $n$ under $f$}. If there exist positive integers $K$ and $c$ such that $f^k(n)=c$ for all integers $k\geq K$, then we say the trajectory of $n$ under $f$ \emph{collapses} to $c$. If the trajectory of $n$ under $f$ collapses to $c$, then the \emph{iteration length of $n$ under $f$} is the smallest positive integer $k$ such that $f^k(n)=c$. For example, the trajectory of any positive integer under $\varphi$ collapses to $1$. The iteration lengths of $1$ and $2$ under $\varphi$ are both $1$, while the iteration length of $100$ under $\varphi$ is $6$. 
\par 
Inspired by the Euler $\varphi$ function, White defined a function of \emph{finite index} to be any arithmetic function $f$ such that the trajectory of any positive integer under $f$ collapses to $1$ \cite{white62}. White then proved that any multiplicative arithmetic function satisfying 
\begin{equation} \label{Eq1}
q\vert f(p^{\alpha})\Rightarrow q\leq p \text{ for all } p,q\in\mathbb{P} \text{ and }\alpha\in\mathbb{N}
\end{equation} and 
\begin{equation} \label{Eq2}
p^{\alpha}\nmid f(p^{\alpha})\text{ for all }p\in\mathbb{P}\text{ and }\alpha\in\mathbb{N} 
\end{equation} is of finite index. It is clear that $\varphi$ satisfies \eqref{Eq1} and \eqref{Eq2}. The function $\psi$ satisfies \eqref{Eq2}, but does not satisfy \eqref{Eq1} in the case $p=2$, $q=3$. This makes perfect sense because $\psi(n)>n$ for all integers $n>1$, which implies that $\psi$ cannot be of finite index. In fact, it is easy to show that the trajectory of any integer $n>1$ under $\psi$ has a tail of the form $2^a3^b,2^{a+1}3^b,2^{a+2}3^b,\ldots$ for some $a,b\in\mathbb{N}$. We choose to study $\overline{\psi}$ instead of $\psi$ because $\overline{\psi}$ does satisfy \eqref{Eq1} and \eqref{Eq2} (and is, therefore, of finite index). For $n>1$, we let $\lambda(n)$ denote the unique nonnegative integer satisfying $\overline{\psi}^{\lambda(n)}(n)=2$ (it is easy to see that such an integer must exist and be unique), and we define $\lambda(1)=0$. Therefore, $\lambda(n)$ is simply one less than the iteration length of $n$ under $\overline{\psi}$. We say that $n$ is in class $k$ if $\lambda(n)=k$. In order to facilitate calculations, we define a function $D\colon\mathbb{N}\rightarrow\mathbb{N}_0$ by \[D(x)=\begin{cases} \lambda(x), & \mbox{if } x\equiv 1\pmod 2; \\ \lambda(x)+1, & \mbox{if } x\equiv 0\pmod 2. \end{cases}\] It follows from a theorem of White that the function $D$ is completely additive \cite[Theorem 2]{white62}. That is, $D(xy)=D(x)+D(y)$ for all $x,y\in\mathbb{N}$. In particular, observe that $D(2^{\alpha})=\alpha$ for all positive integers $\alpha$. We now proceed to determine some properties of the classes into which the function $\lambda$ partitions $\mathbb{N}$. 
\section{The Structures of Classes}
We begin by determining special values that provide bounds for the numbers in a given class. Define a function $g\colon\mathbb{N}\backslash\{1\}\rightarrow\mathbb{N}$ by \[g(k)=\begin{cases} 5^{k/3}, & \mbox{if } k\equiv 0\pmod 3; \\ 9\cdot 5^{(k-4)/3}, & \mbox{if } k\equiv 1\pmod 3; \\ 3\cdot 5^{(k-2)/3}, & \mbox{if } k\equiv 2\pmod 3. \end{cases}\]
\begin{lemma} \label{Lem2.1} 
For all $m_1,m_2\in\mathbb{N}\backslash\{1\}$, we have $g(m_1)g(m_2)\geq g(m_1+m_2)$. 
\end{lemma} 
\begin{proof} 
If $m_1\equiv m_2\equiv 0\pmod 3$, then $g(m_1)g(m_2)=5^{m_1/3}\cdot 5^{m_2/3}=$ \\$5^{(m_1+m_2)/3}=g(m_1+m_2)$. If $m_1\equiv m_2\equiv 1\pmod 3$, then $g(m_1)g(m_2)=81\cdot5^{(m_1+m_2-8)/3}>3\cdot5^{(m_1+m_2-2)/3}=g(m_1+m_2)$. The other cases are handled similarly. 
\end{proof} 
\begin{lemma} \label{Lem2.2}
There are no odd numbers in class $1$. Furthermore, if $x$ is an even number in class $m$, then $v_2(x)\neq m$. 
\end{lemma} 
\begin{proof} 
First, note that $1$ is in class $0$. If $n$ is in class $1$, then $\overline{\psi}(n)=2$. This is impossible if $n$ is odd because $\overline{\psi}(n)=\psi(n)>n$ for all odd integers $n>1$. In fact, the only number in class $1$ is $4$. Suppose $x$ is an even positive integer satisfying $v_2(x)=m$. We may write $x=2^mt$ for some odd $t$. Then $\lambda(x)=D(2^mt)-1=D(2^m)+D(t)-1=m+\lambda(t)-1$. Because $t$ is odd, it is not in class $1$. This means that $\lambda(t)\neq 1$, so $\lambda(x)\neq m$.      
\end{proof} 
The following lemma is quite simple, but we record it for future reference. 
\begin{lemma} \label{Lem2.3}
If $n$ is an odd positive integer, then $n$ and $2n$ are in the same class. 
\end{lemma}
\begin{proof}
Observe that $1$ and $2$ are both in class $0$, so the result holds for $n=1$. If $n>1$, then the result is an easy consequence of the fact that $\overline{\psi}(n)=\overline{\psi}(2n)$.  
\end{proof}   
\begin{theorem} \label{Thm2.1} 
Let $k>1$ be an integer. The smallest odd number in class $k$ is $g(k)$, and the smallest even number in class $k$ is $2g(k)$. 
\end{theorem} 
\begin{proof} 
We first use induction to prove that $g(k)$ is indeed in class $k$. We have $\lambda(g(2))=\lambda(3)=2$, $\lambda(g(3))=\lambda(5)=3$, and $\lambda(g(4))=\lambda(9)=4$, so $\lambda(g(k))=k$ for all $k\in\{2,3,4\}$. Let $r>4$ be an integer, and suppose that $\lambda(g(k))=k$ if $k\in\{2,3,\ldots,r-1\}$.  Because $g(r)$ is odd, $\lambda(g(r))=D(g(r))=D(5g(r-3))=D(5)+D(g(r-3))=D(5)+r-3=r$. By induction, we see that $g(k)$ is in class $k$ for any integer $k>1$. Therefore, $2g(k)$ is also in class $k$ by Lemma \ref{Lem2.3}.  
\par 
It is easy to verify that the theorem holds if $k\in\{2,3,4\}$. Therefore, let $m>4$ be a positive integer, and suppose the theorem holds whenever $k<m$. Let $n$ and $s$ be the smallest odd number and the smallest even number, respectively, in class $m$. We have already shown that $n\leq g(m)$ and $s\leq 2g(m)$, so we wish to prove the opposite inequalities. Assume, first, that $n$ is a prime. Then $m=\lambda(n)=1+\lambda(\overline{\psi}(n))=1+\lambda(n+1)$, so $n+1$ is in class $m-1$. By the induction hypothesis and the fact that $n+1$ is even, we have $n+1\geq 2g(m-1)$. It is easy to see from the definition of $g$ that $2g(m-1)\geq g(m)+1$, so $n\geq g(m)$. Assume, now, that $n$ is composite. We may write $n=n_1n_2$, where $n_1$ and $n_2$ are odd integers satisfying $1<n_1\leq n_2<n$. Write $\lambda(n_1)=m_1$ and $\lambda(n_2)=m_2$. We have $m=\lambda(n)=D(n)=D(n_1)+D(n_2)=m_1+m_2$. We know that $m_1$ and $m_2$ are each less than $m$ because $n$ was assumed to be the smallest odd number in class $m$ (or, alternatively, because $m_1$ and $m_2$ must be positive because $n_1$ and $n_2$ are each greater than $2$). By the induction hypothesis, $n_1\geq g(m_1)$, and $n_2\geq g(m_2)$. Using Lemma \ref{Lem2.1}, we may conclude that $n=n_1n_2\geq g(m_1)g(m_2)\geq g(m_1+m_2)=g(m)$. This proves that $n=g(m)$. 
\par 
Now, write $s=2^{\alpha}t$, where $t$ is odd. We have $m=\lambda(2^{\alpha}t)=D(2^{\alpha}t)-1=D(2^{\alpha})+D(t)-1=\lambda(t)+\alpha-1$, so $\lambda(t)=m-(\alpha-1)$. If $t=1$, then $\alpha=m+1>5$. However, it is easy to see that $2^{\alpha-2}\cdot3$ would then be an even integer in class $m$ that is smaller than $s$. This is a contradiction, so $t>1$. Lemma \ref{Lem2.2} guarantees that $m-(\alpha-1)>1$ because $t$ is an odd number in class $m-(\alpha-1)$. Because we know from the induction hypothesis (and the preceding paragraph in the case $\alpha=1$) that the smallest odd number in class $m-(\alpha-1)$ is $g(m-(\alpha-1))$, we have $t\geq g(m-(\alpha-1))$. If $\alpha=1$, then we have obtained the desired inequality $s=2t\geq 2g(m)$. As mentioned in the preceding paragraph, $2g(m-1)>g(m)$. Hence, if $\alpha=2$, then $s=2^2t\geq 2^2g(m-1)>2g(m)$. Finally, suppose $\alpha>2$. It is easy to see that $g(\alpha-1)\leq 2^{\alpha-1}$. Using Lemma \ref{Lem2.1}, we have $g(\alpha-1)g(m-(\alpha-1))\geq g(m)$. Hence, $s=2\cdot2^{\alpha-1}t\geq 2\cdot2^{\alpha-1}g(m-(\alpha-1))\geq2g(\alpha-1)g(m-(\alpha-1))\geq2g(m)$.   
\end{proof} 
We may use Theorem \ref{Thm2.1} to prove the following more general result.  
\begin{theorem} 
Let $a$ and $k$ be positive integers with $k>\lambda(a)+1$. The smallest multiple of $a$ in class $k$ is $ag(k-\lambda(a))$. If $a$ is odd, then the only multiple of $a$ in class $\lambda(a)+1$ is $4a$. If $a$ is even, then the only multiple of $a$ in class $\lambda(a)+1$ is $2a$. 
\end{theorem} 
\begin{proof} 
First note that $\lambda(ag(k-\lambda(a)))=\lambda(a)+\lambda(g(k-\lambda(a)))=k$ because $g(k-\lambda(a))$ is an odd number in class $k-\lambda(a)$. Hence, $ag(k-\lambda(a))$ is indeed in class $k$. Let $am$ be a multiple of $a$ that is in class $k$. We wish to show that $am\geq ag(k-\lambda(a))$. Suppose that $a$ is odd or $m$ is odd. Then $\lambda(m)=\lambda(am)-\lambda(a)=k-\lambda(a)$. This means that $m$ is in class $k-\lambda(a)$, so $m\geq g(k-\lambda(a))$ by Theorem \ref{Thm2.1}. Thus, $am\geq ag(k-\lambda(a))$. Suppose, now, that $a$ and $m$ are both even. Then $\lambda(m)=D(m)-1=D(am)-D(a)-1=\lambda(am)-D(a)=k-D(a)=k-\lambda(a)-1$, so $m$ is an even number in class $k-\lambda(a)-1$. If $k=\lambda(a)+2$, then we must have $m=4$ because $4$ is the only number in class $1$. In this case, $am=4a>3a=ag(k-\lambda(a))$. 
Thus, let us suppose $k>\lambda(a)+2$. Theorem \ref{Thm2.1} yields the inequality $m\geq 2g(k-\lambda(a)-1)$. Recall that $2g(x-1)>g(x)$ for any integer $x>1$, so $m>g(k-\lambda(a))$. This completes the proof of the first statement of the theorem. 
\par 
Suppose $a$ is odd and $at$ is a multiple of $a$ in class $\lambda(a)+1$. Then $\lambda(t)=\lambda(at)-\lambda(a)=1$, so $t=4$. 
\par 
Suppose $a$ is even and $au$ is a multiple of $a$ in class $\lambda(a)+1$. If $u$ is odd, then $\lambda(u)=\lambda(au)-\lambda(a)=1$, which contradicts Lemma \ref{Lem2.2}. Thus, $u$ must be even. Then $\lambda(u)=D(u)-1=D(au)-D(a)-1=\lambda(au)-D(a)=\lambda(a)+1-D(a)=0$. This implies that $u=2$ because $2$ is the only even number in class $0$.    
\end{proof} 
\begin{theorem} \label{Thm2.2}
Let $k>1$ be an integer. The largest even number in class $k$ is $2^{k+1}$, and the largest odd number in class $k$ is less than $2^k$.  
\end{theorem} 
\begin{proof} Observe that if $n>1$, then $\displaystyle{\overline{\psi}(n)\geq \frac{1}{2}n}$. Similarly, if $\displaystyle{\overline{\psi}(n)>1}$, then $\displaystyle{\overline{\psi}^2(n)\geq \frac{1}{4}n}$. Continuing this same argument, we find that $2=\overline{\psi}^{\lambda(n)}(n)\geq \displaystyle{\frac{1}{2^{\lambda(n)}}n}$. Hence, if $n$ is in class $k$, then $n\leq 2^{k+1}$. Because $2^{k+1}$ is in class $k$, this proves that the largest even number in class $k$ is $2^{k+1}$. Suppose $n>1$ is odd. Then $\overline{\psi}(n)\geq n+1$, so $\overline{\psi}^2(n)\geq \displaystyle{\frac{1}{2}(n+1)}$. Again, we continue this argument until we find that $2=\overline{\psi}^{\lambda(n)}(n)\geq \displaystyle{\frac{1}{2^{\lambda(n)-1}}(n+1)}$. This shows that if $n$ is in class $k$, then $n+1\leq 2^k$. This proves the second claim.     
\end{proof} 
For $k>1$, Theorems \ref{Thm2.1} and \ref{Thm2.2} allow us to divide class $k$ into three useful sections as follows: \[\underbrace{g(k),\ldots,}_{\text{Section I}}\underbrace{2g(k),\ldots,}_{\text{Section II}}2^k,\underbrace{\ldots,2^{k+1}}_{\text{Section III}}.\] More specifically, Section I of class $k$ is the set of numbers in class $k$ that are less than $2g(k)$. All elements of Section I are necessarily odd by Theorem \ref{Thm2.1}. Section II of class $k$ is the set of elements of class $k$ that are greater than or equal to $2g(k)$ and strictly less than $2^k$. Section II of a class may contain both even and odd elements. Finally, Section III of class $k$ is the set of numbers in class $k$ that are greater than $2^k$; all such numbers are necessarily even by Theorem \ref{Thm2.2}. We must define the sections of class $0$ and class $1$ explicitly. We will say that Section I of class $0$ contains only the number $1$, and Section III of class $0$ contains only the number $2$. We will say that Section II of class $0$ is empty. Similarly, Sections I and II of class $1$ will be empty while Section III of class $1$ will contain only the number $4$ (the only number in class $1$).
\par 
Let $S_1$ denote the set of all positive integers that are in Section I of their respective classes, and define $S_2$ and $S_3$ similarly for numbers in Section II and Section III, respectively. For each nonnegative integer $k\neq 1$, let $B(k)$ denote the largest odd number in class $k$ (observe that such a number must exist because $1$ is in class $0$ and $g(k)$ is an odd number in class $k$ when $k>1$). Let $\mathcal B=\{B(k)\colon k\in\mathbb{N}_0\backslash\{1\}\}$. Theorem \ref{Thm2.2} states that $B(k)<2^k$ for each integer $k>1$, but it does not give any lower bound for $B(k)$. This leads us to the following lemma and theorem. 
\begin{lemma} \label{Lem2.4}
If $0<p<q$, then $(2^q-1)^p>(2^p-1)^q$. 
\end{lemma} 
\begin{proof} 
Let $f(x)=(2^x-1)^{1/x}$. Then \[f'(x)=(2^x-1)^{1/x}\left(\frac{2^x\log 2}{x(2^x-1)}-\frac{\log(2^x-1)}{x^2}\right)\] \[>(2^x-1)^{1/x}\left(\frac{2^x\log 2}{x2^x}-\frac{\log(2^x)}{x^2}\right)=0.\] Therefore, if $0<p<q$, then $f(p)<f(q)$. This then implies that $(2^q-1)^p=f(q)^{pq}>f(p)^{pq}=(2^p-1)^q$. 
\end{proof} 
\begin{theorem} \label{Thm2.3} 
Let $2^p-1$ and $2^q-1$ be Mersenne primes with $p<q$. If $k\geq(p-1)(q-1)$, then $B(k)\geq 2^k(1-2^{-q})^{\frac{k-p(q-1)}{q}}(1-2^{-p})^{q-1}$.  
\end{theorem} 
\begin{proof} 
Fix some $k\geq (p-1)(q-1)$. A result due to Skupie\'n \cite[Prop. 3.2]{skupien93} implies that there exist unique positive integers $a$ and $b$ with $b\leq q-1$ such that $k=aq+bp$. Let $j=q-1-b$. Then $\displaystyle{a=\frac{k-p(q-1)}{q}+\frac{p}{q}j}$. Observe that $\lambda(2^p-1)=1+\lambda(\overline{\psi}(2^p-1))=1+\lambda(2^p)=p$. Similarly, $\lambda(2^q-1)=q$. Therefore, \[\lambda((2^q-1)^a(2^p-1)^b)=aq+bp=k.\] Because $(2^q-1)^a(2^p-1)^b$ is an odd number in class $k$, we have \[B(k)\geq (2^q-1)^a(2^p-1)^b=2^k(1-2^{-q})^a(1-2^{-p})^b\] \[=2^k(1-2^{-q})^{\frac{k-p(q-1)}{q}+\frac{p}{q}j}(1-2^{-p})^{q-1-j}=(1-2^{-q})^{\frac{p}{q}j}(1-2^{-p})^{-j}L,\] where $L=2^k(1-2^{-q})^{\frac{k-p(q-1)}{q}}(1-2^{-p})^{q-1}$. From Lemma \ref{Lem2.4}, we have \[(2^q-1)^p>(2^p-1)^q,\] so \[(1-2^{-q})^{\frac{p}{q}j}(1-2^{-p})^{-j}L=((2^q-1)^{p})^{\frac{j}{q}}2^{-pj}(1-2^{-p})^{-j}L\] \[>(2^p-1)^j2^{-pj}(1-2^{-p})^{-j}L=L.\]  
\end{proof}  
The numbers $B(k)$ are quite difficult to handle, so a better lower bound currently evades us. Numerical evidence seems to suggest that the elements of $\mathcal B$ (other than $1$) are all primes, but we have no proof or counterexample. We may, however, prove the following fact quite easily. The analogue of the following theorem for Euler's totient function was originally proven by Catlin \cite{catlin70}.
\begin{theorem} \label{Thm2.4} 
If $x\in\mathcal B$, then every positive divisor of $x$ is in $\mathcal B$. 
\end{theorem} 
\begin{proof} 
Let $x\in \mathcal{B}$, and suppose $d$ is a positive divisor of $x$ that is not in $\mathcal{B}$. As $x\in\mathcal B$, $x$ must be odd. This implies that $d$ is odd, so there must be some odd number $m>d$ that is in the same class as $d$. Write $x=dy$. Then $\lambda(x)=\lambda(d)+\lambda(y)=\lambda(m)+\lambda(y)=\lambda(my)$. Then $my$ is an odd number that is in the same class as $x$ and is larger than $x$. This is a contradiction, and the desired result follows. 
\end{proof} 
We may prove a similar result for the numbers in $S_1$. The analogue of the following theorem for Euler's totient function was originally proven by Shapiro \cite{shapiro43}. 
\begin{theorem} \label{Thm2.5}
If $x\in S_1$, then every positive divisor of $x$ is in $S_1$.
\end{theorem} 
\begin{proof} 
Let $x\in S_1$, and let $d$ be a positive divisor of $x$. Write $x=dy$. We wish to show that $d$ is in Section I of its class. This is trivial if $d=1$ or $d=x$, so we may assume that $d$ and $y$ are both greater than $1$. As $x\in S_1$, $x$ must be odd. This implies that $d$ and $y$ are odd, so $\lambda(x)=\lambda(d)+\lambda(y)$. Lemma \ref{Lem2.1} then implies that $g(\lambda(d))g(\lambda(y))\geq g(\lambda(x))$. The fact that $x\in S_1$ implies that $x<2g(\lambda(x))$. Furthermore, $y\geq g(\lambda(y))$ by Theorem \ref{Thm2.1}, so $d=x/y<2g(\lambda(x))/y\leq 2g(\lambda(x))/g(\lambda(y))\leq 2g(\lambda(d))$. Hence, $d$ is in Section I of its class.   
\end{proof}
We next determine a bit of information about the structure of $S_3$. We will do so in greater generality than is necessary, and the results pertinent to $S_3$ with follow as corollaries. To this end, define $V(c)=\{n\in\mathbb{N}\colon n>2^{\lambda(n)-c}\}$ for any real number $c$, and observe that $S_3=V(0)$. Also, define $\eta(p)=\lambda(p)-\log_2(p)$ for all odd primes $p$. Note that $\eta(p)$ is always positive by Theorem \ref{Thm2.2}. 
\begin{lemma} \label{Lem2.5} 
Let $r$ and $\alpha$ be positive integers, and let $c$ be a real number. Let $t=p_1^{\alpha_1}\cdots p_r^{\alpha_r}$, where $p_1,\ldots,p_r$ are distinct odd primes and $\alpha_1,\ldots,\alpha_r$ are positive integers. Then $2^{\alpha}t\in V(c)$ if and only if $\displaystyle{\sum_{i=1}^r}\alpha_i\eta(p_i)<c+1$.     
\end{lemma} 
\begin{proof} 
By the definition of $V(c)$, $2^{\alpha}t\in V(c)$ if and only if $2^{\alpha}t>2^{\lambda(2^{\alpha}t)-c}$. Now, $\lambda(2^{\alpha}t)=\alpha-1+\displaystyle{\sum_{i=1}^r\alpha_i\lambda(p_i)}$, so $2^{\lambda(2^{\alpha}t)-c}=2^{\alpha-c-1}\displaystyle{\prod_{i=1}^r2^{\alpha_i\lambda(p_i)}}$. Therefore, $2^{\alpha}t\in V(c)$ if and only if $2^{c+1}t>\displaystyle{\prod_{i=1}^r2^{\alpha_i\lambda(p_i)}}$. This inequality is equivalent to \[c+1+\log_2(p_1^{\alpha_1}\cdots p_r^{\alpha_r})>\log_2\left(\prod_{i=1}^r2^{\alpha_i\lambda(p_i)}\right),\] which we may rewrite as $\displaystyle{\sum_{i=1}^r\alpha_i\eta(p_i)<c+1}$.
\end{proof} 
Although the proof of Lemma \ref{Lem2.5} is quite simple, the lemma tells us quite a bit about the factors of elements of $V(c)$. For example, we see that whether or not an even integer is an element of $V(c)$ is independent of the power of $2$ in its prime factorization. Also, any even divisor of an even element of $V(c)$  must itself be in $V(c)$. In the case $c=0$, this provides an analogue of Theorem \ref{Thm2.5} for $S_3$. Furthermore, if $q$ is an odd prime and $k=\displaystyle{\left\lfloor\frac{c+1}{\eta(q)}\right\rfloor}$, then Lemma \ref{Lem2.5} tells us that $2^{\alpha}q^{\beta}\in V(c)$ for all positive integers $\alpha$ and $\beta$ with $\beta\leq k$. On the other hand, $q^{\beta}$ does not divide any even element of $V(c)$ if $\beta>k$. In particular, if $\eta(q)>c+1$ so that $k=0$, then no even element of $V(c)$ is divisible by $q$. This leads us to inquire about the set $T(c)$ of odd primes $q$ satisfying $\eta(q)>c+1$. Let $b_c(x)=\vert\{q\in T(c)\colon q\leq x\}\vert$. One may wish to determine estimates for $b_c(x)$. We have not attempted to do so in great detail; rather, we will simply show that $\displaystyle{\lim_{x\rightarrow\infty}\frac{b_c(x)}{\pi(x)}}=1$. In other words, $T(c)$ has asymptotic density $1$ when regarded as a subset of the set of primes.    
\begin{lemma} \label{Lem2.6} 
Let $c$ be a real number. Let $q\in T(c)$, and let $p$ be an odd prime such that $p\equiv -1\pmod q$. Then $p\in T(c)$. 
\end{lemma}
\begin{proof} 
As mentioned in the above discussion, $q$ cannot divide any even element of $V(c)$. Hence, $p+1\not\in V(c)$, which implies that $p+1\leq 2^{\lambda(p+1)-c}$. Therefore, $\lambda(p)=1+\lambda(\overline{\psi}(p))=1+\lambda(p+1)\geq c+1+\log_2(p+1)$, so $\eta(p)=\lambda(p)-\log_2(p)>c+1$.
\end{proof} 
We will also need the following fact, the proof of which follows as an easy consequence of the well-known generalization of Dirichlet's Theorem that states that, for any relatively prime positive integers $k$ and $l$, $\displaystyle{\sum_{\substack{p\leq x \\ p\equiv l\pmod k}}\frac{1}{p}}$ diverges as $x\rightarrow\infty$. 
\begin{fact} \label{Fact2.1} 
For any relatively prime positive integers $k$ and $l$, \[\lim_{x\rightarrow\infty}\left[\prod_{\substack{p\leq x \\ p\equiv l\pmod k}}\left(1-\frac{1}{p}\right)\right]=0.\]
\end{fact} 
\begin{theorem} \label{Thm2.6} 
Let $c$ be a real number. With $b_c$ defined as above, we have $\displaystyle{\lim_{x\rightarrow\infty}\frac{b_c(x)}{\pi(x)}=1}$.  
\end{theorem} 
\begin{proof}
If $c\leq -1$, then the result is trivial because every odd prime is in $T(c)$. Therefore, let us assume that $c>-1$. We first need to show that $T(c)$ is nonempty. Let $k=\displaystyle{\left\lfloor\frac{c+1}{\eta(3)}\right\rfloor=\left\lfloor\frac{c+1}{2-\log_2(3)}\right\rfloor}$. By Dirichlet's Theorem, there exist infinitely many primes in the arithmetic progression $2\cdot3^{k+1}-1,4\cdot3^{k+1}-1,6\cdot3^{k+1}-1,\ldots$. Let $\omega$ be one such prime, and write $\omega=2\ell\cdot3^{k+1}-1$. If we replace $2^{\alpha}t$ with $2\ell\cdot3^{k+1}$ in Lemma \ref{Lem2.5}, then we see that $\displaystyle{\sum_{i=1}^r\alpha_i\eta(p_i)\geq (k+1)\eta(3)}$, where we have used the notation from Lemma \ref{Lem2.5}. By the definition of $k$, $(k+1)\eta(3)>c+1$. Lemma \ref{Lem2.5} then tells us that $2\ell\cdot3^{k+1}=\omega+1\not\in V(c)$. Hence, $\omega+1\leq 2^{\lambda(\omega+1)-c}=2^{\lambda(\overline{\psi}(\omega))-c}=2^{\lambda(\omega)-c-1}$. This yields the inequality $\log_2(\omega+1)\leq \lambda(\omega)-c-1$, which implies that $\eta(\omega)=\lambda(\omega)-\log_2(\omega)>c+1$. Hence, $\omega\in T(c)$, so $T(c)$ is nonempty. 
\par
Now, choose some $\delta,\epsilon>0$. Let $q_0$ be the smallest element of $T(c)$. For any positive real number $y$, let $Q(y)$ be the set of odd primes that are less than or equal to $y$ and that are congruent to $-1$ modulo $q_0$. By Lemma \ref{Lem2.6}, $Q(y)\subseteq T(c)$. Let us choose $y$ large enough so that $\displaystyle{\prod_{q\in Q(y)}\left(1-\frac{1}{q}\right)}<\delta$ (Fact \ref{Fact2.1} guarantees that we may do so). For any $x>y$, let $A(x,y)$ be the set of odd primes that are less than or equal to $x$ and that are congruent to $-1$ modulo $q$ for some $q\in Q(y)$. Using Lemma \ref{Lem2.6} again, we see that $A(x,y)\subseteq T(c)$. As every element of $A(x,y)$ is less than or equal to $x$, $\vert A(x,y)\vert\leq b_c(x)$. Therefore, if we can show that $\displaystyle{\lim_{x\rightarrow\infty}\frac{\vert A(x,y)\vert}{\pi(x)}=1}$, then we will be done. 
\par 
Let $\mathcal P=\displaystyle{\prod_{q\in Q(y)}q}$. Let $W$ be the set of positive integers less than or equal to $\mathcal P$ that are not congruent to $0$ or $-1$ modulo $q$ for all $q\in Q(y)$. Then $\pi(x)-\vert A(x,y)\vert$ is the number of primes that are less than or equal to $x$ and that are congruent to $w$ modulo $\mathcal P$ for some $w\in W$. That is, $\pi(x)-\vert A(x,y)\vert=\displaystyle{\sum_{w\in W}\pi_w(x)}$, where $\pi_w(x)$ denotes the number of primes less than or equal to $x$ that are congruent to $w$ modulo $\mathcal P$. By the Chinese Remainder Theorem, $\displaystyle{\vert W\vert=\prod_{q\in Q(y)}(q-2)}$. One version of Dirichlet's Theorem tells us that $\displaystyle{\frac{\pi_w(x)}{\pi(x)}=\frac{1}{\varphi(\mathcal P)}(1+o(1))}$ for all $w\in W$. Hence, we may choose $x$ large enough so that $\displaystyle{\frac{\pi_w(x)}{\pi(x)}<\frac{1}{\varphi(\mathcal P)}+\epsilon}$ for all $w\in W$. With such a choice of $x$, we have \[\frac{\pi(x)-\vert A(x,y)\vert}{\pi(x)}=\sum_{w\in W}\frac{\pi_w(x)}{\pi(x)}< \sum_{w\in W}\left(\frac{1}{\varphi(\mathcal P)}+\epsilon\right)=\frac{\vert W\vert}{\varphi(\mathcal P)}+\vert W\vert\epsilon\] 
\[=\prod_{q\in Q(y)}\left(\frac{q-2}{q-1}\right)+\vert W\vert\epsilon<\prod_{q\in Q(y)}\left(1-\frac{1}{q}\right)+\vert W\vert\epsilon<\delta+\vert W\vert\epsilon.\] Because $\vert W\vert$ does not depend on $x$ or $\epsilon$, $\displaystyle{\lim_{x\rightarrow\infty}\frac{\pi(x)-\vert A(x,y)\vert}{\pi(x)}=0}$. Therefore, \[\lim_{x\rightarrow\infty}\frac{\vert A(x,y)\vert}{\pi(x)}=1.\]
\end{proof} 

\end{document}